\documentclass[a4paper,12pt]{amsart}
\usepackage{amssymb}
\usepackage{ifthen}
\usepackage{graphicx}
\usepackage{float}
\nonstopmode
\newtheorem{thm}{Theorem}
\newtheorem{cor}{Corollary}
\newtheorem{lem}{Lemma}

\newtheorem{conj}{Conjecture}
\newtheorem{prob}{Problem}

\theoremstyle{definition}
\newtheorem{defn}{Definition}[section]
\newtheorem{example}{Example}

\newenvironment{rem}{%
\bigskip
\noindent \textsl{{\sl Remark. }}}{\bigskip}
\newenvironment{rems}{%
\bigskip
\noindent \textsl{{\sl Remarks. }}}{\bigskip}

\newenvironment{pf}[1][]{%
 \vskip 1mm
 \noindent
 \ifthenelse{\equal{#1}{}}%
  {{\slshape Proof. }}%
  {{\slshape #1.} }%
 }%
{\qed\medskip}

\newcounter{alphabet}
\newcounter{tmp}

\makeatletter
\newcommand{\Ref}[1]{\@ifundefined{r@#1}{}{\setcounter{tmp}{\ref{#1}}\Alph{tmp}}}
\makeatother

\newcommand{\IR}{{\mathbb R}}
\newcommand{\IN}{{\mathbb N}}
\newcommand{\IC}{{\mathbb C}}
\newcommand{\ID}{{\mathbb D}}

\def\be{\begin{equation}}
\def\ee{\end{equation}}

\newcommand{\bee}{\begin{enumerate}}
\newcommand{\eee}{\end{enumerate}}

\newcommand{\blem}{\begin{lem}}
\newcommand{\elem}{\end{lem}}
\newcommand{\bthm}{\begin{thm}}
\newcommand{\ethm}{\end{thm}}
\newcommand{\bcor}{\begin{cor}}
\newcommand{\ecor}{\end{cor}}
\newcommand{\beg}{\begin{example}}
\newcommand{\eeg}{\end{example}}
\newcommand{\begs}{\begin{examples}}
\newcommand{\eegs}{\end{examples}}
\newcommand{\bdefe}{\begin{defn}}
\newcommand{\edefe}{\end{defn}}
\newcommand{\bprob}{\begin{prob}}
\newcommand{\eprob}{\end{prob}}
\newcommand{\bques}{\begin{ques}}
\newcommand{\eques}{\end{ques}}
\newcommand{\bei}{\begin{itemize}}
\newcommand{\eei}{\end{itemize}}
\newcommand{\bcon}{\begin{conj}}
\newcommand{\econ}{\end{conj}}
\newcommand{\bcons}{\begin{conjs}}
\newcommand{\econs}{\end{conjs}}
\newcommand{\bprop}{\begin{propo}}
\newcommand{\eprop}{\end{propo}}
\newcommand{\br}{\begin{rem}}
\newcommand{\er}{\end{rem}}
\newcommand{\brs}{\begin{rems}}
\newcommand{\ers}{\end{rems}}
\newcommand{\bo}{\begin{obser}}
\newcommand{\eo}{\end{obser}}
\newcommand{\bos}{\begin{obsers}}
\newcommand{\eos}{\end{obsers}}
\newcommand{\bpf}{\begin{pf}}
\newcommand{\epf}{\end{pf}}
\newcommand{\ba}{\begin{array}}
\newcommand{\ea}{\end{array}}
\newcommand{\beq}{\begin{eqnarray}}
\newcommand{\beqq}{\begin{eqnarray*}}
\newcommand{\eeq}{\end{eqnarray}}
\newcommand{\eeqq}{\end{eqnarray*}}

\newcommand{\ra}{\rightarrow}

\newcommand{\ds}{\displaystyle}

\newcounter{minutes}
\divide\time by 60
\newcounter{hours}
\multiply\time by 60 \addtocounter{minutes}{-\time}

\begin{document}
\bibliographystyle{amsplain}
\title[The Jacobian conjecture and injectivity conditions]{The Jacobian conjecture and injectivity conditions}

\thanks{
File:~\jobname .tex,
          printed: \number\day-\number\month-\number\year,
          \thehours.\ifnum\theminutes<10{0}\fi\theminutes}


\author{Saminathan Ponnusamy, and Victor V. Starkov}


\address{S. Ponnusamy, Department of Mathematics,
Indian Institute of Technology Madras, Chennai--600 036, India.}

%
\email{samy@isichennai.res.in, samy@iitm.ac.in}

\address{V. V. Starkov,
Department of Mathematics, Petrozavodsk State University,
ul. Lenina 33, 185910 Petrozavodsk, Russia
}
\email{vstarv@list.ru }


\begin{abstract}
One of the aims of this article is to provide a class of polynomial mappings for which the Jacobian conjecture is true.
Also, we state and prove several global univalence theorems and present a couple of applications of them.
\end{abstract}

\thanks{}

\maketitle
\pagestyle{myheadings}
\markboth{S. Ponnusamy, and V. V. Starkov}{The Jacobian conjecture and injectivity conditions}

\section{Introduction and Main Results}\label{sec1}
This article mainly concerns with mappings $f :\,\IC^n\ra \IC^n$, written in coordinates as
$$ f(Z) =(f_1(Z),\ldots,f_n(Z)), \quad Z= (z_1,\ldots ,z_n).
$$
We say that $f$ is a \textit{polynomial map} if each component function $f_i :\,\IC^n\ra \IC$
is a polynomial in $n$-variables $z_1, \ldots, z_n$, for $1\leq i\leq n$.
A polynomial map $f :\,\IC^n\ra \IC^n$ is called \textit{invertible} if it has an inverse map  which is also a
polynomial map.


Let $Df:=\left (\frac {\partial f_j}{\partial z_i}\right )_{n\times n}$,  $ 1\leq i,j\leq n$, be the Jacobian matrix of $f$.
The Jacobian determinant is denoted by $\det Df$.
If a polynomial map $f$ is invertible and $g=f^{-1}$, then $g\circ f={\rm id}$
and, because $\det Df \cdot \det Dg=1$,  $\det Df$ must be a  non-zero complex constant.
However,  the converse question is more difficult.
Then the Jacobian conjecture (\textbf{JC})
asserts that every polynomial mapping $f:\,\IC^n\ra \IC^n$ is globally invertible
if  $\det Df$ is identically equal to a non-zero complex constant. This conjecture remains open for any dimension $n\ge 2$.
We remark that the \textbf{JC} was originally formulated by
Keller \cite{Keller39} in 1939 for polynomial maps with integer coefficients. In the case of dimension one, it is simple.
Polynomial map $f$ is called a \textit{Keller map}, if $\det Df$ is a non-zero complex constant.
In fact, Bialynicki-Birula and Rosenlicht \cite{BbRo62} proved that a polynomial map
is invertible if it is injective.

It is a simple exercise to see that the \textbf{JC} is true if it holds for polynomial mappings whose Jacobian determinant is $1$ and
thus, after suitable normalization, one can assume that $\det Df= 1$.
The \textbf{JC}  is attractive because of the simplicity of its statement. Moreover, because there are so many ways to approach and making it useful,
the \textbf{JC} has been studied extensively from calculus to complex analysis to algebraic topology, and from commutative algebra to differential
algebra to algebraic geometry. Indeed, some faulty proofs have even been published. The \textbf{JC} is stated as one of the eighteen challenging problems
for the twenty-first century proposed with brief details by Field medalist Steve Smale \cite{smale98}.
For the importance, history,  a detailed account of the research on the \textbf{JC} and equivalent conjectures, and related investigations,
we refer for example to \cite{BaCoWr82} and
the excellent book of van den Essen \cite{Essen00} and the references therein. See also \cite{Dr83,Dr91,Dr91a,Dr93,DrRu85,DrTu92,{Kul-93},Wright81,Wright05}.
We would like to point out that in 1980, Wang \cite{Wang80} showed that every Keller map of degree less than or equal to $2$ is invertible.
In 1982, Bass et. al \cite{BaCoWr82} (see also \cite{Dr83,Yag80}) showed that it suffices to prove the \textbf{JC} for all $n\ge 2$ and all Keller
mappings of the form $f(Z)=Z+H(Z)$, where $Z= (z_1,\ldots ,z_n)$,
and $H$ is cubic homogeneous, i.e., $H=(H_1, \ldots, H_n)$ with $H_i(Z)=(L_i(Z))^3$ and $L_i(Z)=a_{i1}z_1+\cdots +a_{in}z_n$, $1\leq i\leq n$.
Cubic homogeneous map $f$ of this form is called a Dru\.{z}kowski or cubic linear map.
Moreover, polynomial mappings from $\IC^n$ to $\IC^n$ are well behaved than the polynomial mappings from $\IR^n$ to $\IR^n$. Indeed,
Pinchuk \cite{Pinch94} constructed an explicit example to show that there exists
a non-invertible polynomial map $f:\IR^2\rightarrow \IR^2$ with $\det Df(X)\neq 0$ for all $X\in \IR^2$.
In any case, the study of the \textbf{JC} has given rise to several surprising results and interesting relations in various directions
and  in different perspective.
For instance,   Abdesselam \cite{Ab2003} formulated the \textbf{JC} as a question in perturbative quantum field theory and pointed out that
any progress on this question will be beneficial not only for mathematics, but also for
theoretical physics as it would enhance our understanding of perturbation theory.



The main purpose of this work is to identify  the Keller maps for which the \textbf{JC} is true.

\begin{thm}\label{cor4a}
The Jacobian conjecture is true for mappings $F(X)=(A \circ f \circ B)(X),$ where  $X=(x_1,\ldots, x_n)\in \IR^n$,
$A$ and $B$ are linear such that $\det A .\det B \neq0$, $f=(u_1,\ldots,u_n),$
$$u_k(X)=x_k+\gamma_k\left[\alpha_2(x_1+\cdots+x_n)^2+\alpha_3(x_1+\cdots+x_n)^3+\cdots+\alpha_m(x_1+\cdots+x_n)^m\right]
$$
for $k=1,\ldots,n$, $\alpha_j,\gamma_k\in\IR$ with $\sum_{k=1}^n\gamma_k=0$ and $m\in \IN$.
\end{thm}

Often it is convenient to identify $X$ in $\IC^n$ (resp. $\IR^n$) as an $n\times 1$ matrix with entries as complex (resp. real) numbers.
It is interesting to know whether there are other polynomial mappings for which the \textbf{JC} is true in $\IC^n$ (resp. $\IR^n$).
In this connection, we will notice that for the case $n=2$ of Theorem \ref{cor4a} it is possible to prove the following:

\begin{thm}\label{thm5}
With $X=(x,y)\in \IR^2$, consider $f(X)=(u_1(X),u_2(X))$, where $u_k(X)$ for $k=1,2$ are as in Theorem \ref{cor4a} and
$$\tilde{f}(X)=(u_1(X)+W(X),u_2(X)+w(X)),
$$
where $W$ and $w$ are homogeneous polynomials of degree $(m+1)$ in $x$ and $y$. If $\det D{\tilde{f}}(X)\equiv 1$, then
$ \tilde{f}=A^{-1} \circ F\circ A,$
where $A$ is linear homogeneous nondegenerate mapping and
$$F(X)=(u_1(X)+\alpha_{m+1}(x+y)^{m+1}, u_2(X)-\alpha_{m+1}(x+y)^{m+1}),
$$
for some real constant $\alpha_{m+1}$. The Jacobian conjecture is true for the mapping $\tilde{f}$.
\end{thm}

\br
It follows from the proof of Theorem \ref{thm5} that $A$ equals the identity matrix $I$ if $f(X)\not\equiv X$.
\er


In connection with Theorems \ref{cor4a} and \ref{thm5}, it is interesting to note that in the case
$n=2$, the mappings $F$ defined in Theorem \ref{cor4a} provides a complete description of the
Keller mappings $F$ for which ${\rm deg}\, F \leq 3$ (see \cite{star2016}).



Next, we denote by $\widehat{\mathcal{P}}_n(m)$, the set of all polynomial mappings $F:\,\IR^n\rightarrow \IR^n$ of degree less than or equal to $m$
such that  $DF(0)=I$  and $F(0)=0$. Let $P_n(m)$ be a subset consisting of mappings $f\in\widehat{\mathcal{P}}_n(m)$ which
satisfy the conditions of Theorem \ref{cor4a}. Also, we introduce
$$\mathcal{P}_n(m)=\{F\in  \widehat{\mathcal{P}}_n(m):\, \mbox{$F$ is injective}\}.
$$
If $f,g\in P_n(m)$, then
$$f(X)=X+u(X)\gamma ~\mbox{ and }~ g(X)=X+v(X)\delta,
$$
where $X=(x_1,\dots,x_n)$, $\gamma =(\gamma_1,\dots,\gamma_n)$, $\delta=(\delta_1,\dots,\delta_n)$,
$$u(X)=\sum_{k=2}^m\alpha_k(x_1+\dots +x_n)^k
~\mbox{ and }~
v(X)=\sum_{k=2}^m\beta_k(x_1+\dots +x_n)^k, 
$$
such that $\sum_{k=1}^n\delta_k=\sum_{k=1}^n\gamma_k=0$. Here $\alpha_k$'s and $\beta_k$'s are some constants.
It is obvious that $f\circ g$ is injective but it is unexpected that the composition $f\circ g$ also belongs to $\mathcal{P}_n(m)$.
This circumstance allows us to generalize Theorem \ref{cor4a} significantly.

\begin{thm}\label{thm1a}
For $n\ge 3$, consider the mapping $f:\,\IR^n\rightarrow \IR^n$ defined by $f(X)=(u_1,\ldots,u_n)$, where
$$u_k(X)=x_k+p_k^{(2)}(x_1+\dots +x_n)^2+\cdots+p_k^{(m)}(x_1+\dots +x_n)^m~ \quad (k=1,\ldots, n),
$$
and $p_k^{(l)}$ are constants satisfying the condition $\sum_{k=1}^np_k^{(l)}=0$ for all $l=2,\ldots,m$.
Then we have  $f\in\mathcal{P}_n(m)$.
\end{thm}

From the proofs of Theorems \ref{cor4a},  \ref{thm5} and \ref{thm1a},  it is easy to see that these results continue to hold even
if we replace $\IR^n$ by $\IC^n$. The proofs of Theorems \ref{cor4a},  \ref{thm5} and \ref{thm1a} will be presented in
Section  \ref{sec3}. In Section \ref{sec2}, we present conditions for injectivity of functions defined on a convex domain.

\section{Injectivity conditions on convex domains} \label{sec2}

One can find discussion and several sufficient conditions for global injectivity \cite{AvkAks75,TP83}. In the following, we state
and prove several results on injectivity on convex domains.

\bthm\label{GPS3-lem1}
Let $D\subset\IR^n$ be convex and $f:\, D\rightarrow\IR^n$ belong to $C^1(D)$. Then $f=(f_1,\ldots,f_n)$ is injective in $D$ if
for every $X_1, X_2\in D$ $(X_1\neq X_2)$ and $\gamma (t)=X_1+t(X_2-X_1)$ for $t\in[0,1]$, $\det A\neq 0,$ where
$$A=(a_{ij})_{n\times n} ~\mbox{ with }~a_{ij}=\int_0^1 \frac {\partial f_j}{\partial x_i} (\gamma (t))\,dt, \, 1\leq i,j\leq n .
$$
\ethm
\begin{proof}
Let $X_1,X_2\in D$ be two distinct points.
Since $D$ is convex, the line segment $\gamma(t)\in D$ for $t\in (0,1)$ and thus, we have
$$
f(X_2)-f(X_1)  =\int_\gamma d f(\zeta ) = \int_\gamma D f(\zeta )\,d\zeta  = \int_0^1 (Df)(\gamma (t))(X_2-X_1)\,dt .
$$
Taking into account of the assumptions, we deduce that $f(X_2)-f(X_1)\neq0$ for each $X_1,X_2\in D$ $(X_1\neq X_2)$ if  
for every $X\in \IR^n \backslash \{0\},$
$$\int_0^1 Df(\gamma(t)) \,dt\cdot X\neq 0
$$
which holds whenever $\det A\neq 0.$ The proof is complete.
\end{proof}

We remark that Theorem \ref{GPS3-lem1} has obvious generalization for functions defined on convex domains $D \subseteq  \IC^n$.
In this case,  the Jacobian matrix of $f$, i.e. $Df=\left (\frac {\partial f_j}{\partial z_i}\right )_{n\times n}$,  $ 1\leq i,j\leq n$,
will be used.  Moreover, using Theorem \ref{GPS3-lem1}, we may easily obtain the
following simple result.

\begin{cor}\label{cor1}
Let $D \subseteq  \IR^n$ be a convex domain and $f=(f_1,\ldots,f_n)$ belong to $C^1(D)$. If for every line $L$ in $\IR^n$, with
$L\cap D\neq\emptyset,$
$$\det \left (\frac{\partial f_j}{\partial x_i}(X_{i,j})\right )\neq0 ~\mbox{for every $ X_{i,j}\in L \cap D$,}
$$
then $f$ is injective in $D$.
\end{cor}
\begin{proof}
Let $X_1,X_2\in D$ be two distinct points and $\gamma (t)=X_1+t(X_2-X_1)$ be the line segment joining $X_1$ and $X_2$, $t\in [0,1]$.
Then, for every $i, j=1,\ldots, n$, we have
$$\int_0^1 \frac {\partial f_j}{\partial x_i} (\gamma (t))\,dt=\frac {\partial f_j}{\partial x_i} (X_{i,j}),
$$
where  $X_{i,j}\in (X_1,X_2)$. The desired conclusion follows from Theorem \ref{GPS3-lem1}.
\end{proof}

\begin{cor}\label{cor2}
Let $D \subseteq  \IC$ be a convex domain, $z=x+iy$ and $f(z)=u(x,y)+iv(x,y)$ be analytic in $D$. Then $f$ is injective in $D$ whenever
for every $z_1,z_2\in D$, we have
$$\det\begin{bmatrix}u_x(z_1)&-v_x(z_2)\\v_x(z_2)& u_x(z_1)\end{bmatrix} =u_x^2(z_1)+v_x(z_2)^2\neq0 .
$$
\end{cor}

Corollary \ref{cor2} shows that if $u_x\ne 0$ or $v_x\ne 0$ in a convex domain $D$, then the analytic function $f=u+iv$ is univalent 
(injective) in $D$.
Thus, it is a sufficient condition for the univalency and is different from the necessary condition $f'(z)\neq 0$,
the fact that in the latter case $u_x$ and $v_x$ have no common zeros in $D$. The reader may compare with the
well-known Noshiro-Warschawski theorem which asserts that if $f$ is analytic in a convex domain $D$ in $\IC$ and ${\rm Re}\, f'(z)>0$
in $D$, then $f$ is univalent in $D$. See also Corollary \ref{GPS2-cor1}.

Throughout we let  $\IR^n_{+}=\{X=(x_1,\ldots, x_n)\in \IR^n:\, x_1>0\}$ and $S_{n-1}=\{X \in \IR^n:\, \|X\|=1\}$, the unit sphere in the
Euclidean space $\IR^n$.

\begin{thm}\label{thm1}
Let $D$ $\subseteq \IR^n$ be a convex domain, $f:\,D\rightarrow \IR^n$ belong to $C^1(D)$ and $ \Omega=f(D)$. Then $f$ is injective if and only if there exists a
$\phi\in C^1(\Omega),$ $\phi:\,\Omega\rightarrow\IR^n,$ satisfying the following property: for every $X_0\in S_{n-1}$ there exists a
unitary matrix $U=U(X_0)$ with
\begin{equation}\label{equ5}
U\cdot (D\phi)(f(X))\cdot Df(X)X_0\in \IR^n_+
\end{equation}
for every $X\in D$.
\end{thm}
\begin{proof}
Let $X_1,X_2\in D$.  Then the line segment $\gamma(t)$ connecting these points
given by $\gamma(t)=X_1+t(X_2-X_1)$ belongs to the convex domain $D$ for every $t\in[0,1]$. We denote $\psi=\phi\circ f$ and observe that
$$\psi(X_2)-\psi(X_1)=\int_\gamma d\psi(\zeta )=\int_0^1d(\psi \circ \gamma)(t)=\int_0^1 (D\psi)(\gamma(t))\cdot(X_2-X_1)\,dt.
$$
If $X_2\neq X_2,$ we may let
$$X_0=\frac{X_2-X_1}{||X_2-X_1||} \in S_{n-1}.
$$

Sufficiency ($\Leftarrow$): Now we assume \eqref{equ5} and show that $f$ is injective on $D$. Because of the truth of \eqref{equ5}, it follows that
$$U(\psi(X_2)-\psi(X_1))=\|X_2-X_1\|\int_0^1U(D\psi)(\gamma(t))X_0\,dt\neq 0 .
$$
Then the first component $a_1(t)$ of
$U(D\psi)(\gamma(t))X_0=(a_1(t),\ldots,a_n(t))\in \IR^n_+$,
by definition, satisfies the positivity condition $a_1(t)>0$ for each $t$ and thus,  $\int_0^1a_1(t)\,dt\neq0$.
Consequently, $f(X_2)\neq f(X_1)$ for every $X_1,X_2\, (X_1\neq X_2)$ in $D$.

Necessity ($\Rightarrow $): Assume that $f$ is injective in $D$. Then we may let $\phi=f^{-1}$ and assume that $U$ is an unitary matrix such that
$UX_0=(1,0,\dots,0)$. This implies that $\psi(X)\equiv X$ and thus,
$$U\cdot (D\psi)(\gamma(t))X_0\equiv(1,0,\ldots,0)
$$
and \eqref{equ5} holds.
\end{proof}

It is now appropriate to state a several complex variables analog of Theorem \ref{thm1}. As with standard practice,
for $D\subset \IC^n$, we consider $f\in C^1(D),$ $\Omega=f(D)\subset\IC^n,$ $\phi\in C^1(\Omega),$ $\phi:\,\Omega\longrightarrow\IC^n$,
$Z=(z_1,\ldots,z_n)\in\IC^n,$ $\overline{Z}=(\bar{z}_1,\ldots,\bar{z}_n)$, and $\psi=\phi \circ f$. We frequently, write down these
mappings as functions of the independent complex variables $Z$ and $\overline{Z}$, namely, as $f(Z,\overline{Z})$,
$\phi(W,\overline{W})$ and $\psi(Z,\overline{Z})$. Denote as usual
$$\partial\psi= \frac{\partial(\psi_1,\ldots,\psi_n)}{\partial( {z}_1,\ldots, {z}_n)} =\begin{pmatrix} \ds
\frac{\partial\psi_1}{\partial z_1} &\cdots &\ds \frac{\partial\psi _1}{\partial z_n}\\
\vdots & \cdots &\vdots\\
\ds \frac{\partial\psi_n}{\partial z_1}&\cdots & \ds \frac{\partial\psi_n}{\partial z_n}\\
\end{pmatrix}
~\mbox{and }~\overline{\partial}\psi =\frac{\partial(\psi_1,\ldots,\psi_n)}{\partial(\bar{z}_1,\ldots,\bar{z}_n)}.
$$
At this place it is convenient to use $\partial\psi$ instead of $D\psi$.
Then for $\gamma(t)=Z_1+t(Z_2-Z_1)$, $t\in(0,1)$, we have
\beqq
\psi(Z_2,\overline{Z_2})-\psi(Z_1,\overline{Z_1})&=&\int_\gamma d\psi(Z,\overline{Z})\\
&=&\int_0^1\left[\partial\psi(\gamma(t),\overline{\gamma(t)})\cdot (Z_2-Z_1)
+\overline{\partial}\psi(\gamma(t),\overline{\gamma(t)})\cdot\overline{(Z_2-Z_1)}\right]dt.
\eeqq
Thus, Theorem \ref{thm1} takes the following form.

\begin{thm}\label{thm2}
Let $D\subseteq\IC^n$ be a convex domain, $f:\,D\rightarrow\IC^n$ belong to $C^1(D)$ and $\Omega=f(D)$. Then $f$ is injective in $D$
if and only if there exists a $\phi\in C^1(\Omega),$ $\phi:\,\Omega\rightarrow\IC^n$ satisfying the following property with $\psi=\phi \circ f$:
for every $Z_0\in \IC^n$, $\|Z_0\|=1$, there exists a unitary complex matrix $U=U(Z_0)$ such that for
every $Z\in D,$ one has
$$U\left[\partial\psi(Z,\overline{Z})Z_0+\overline{\partial}\psi (Z,\overline{Z})\overline{Z}_0\right]\in\{Z\in\IC^n:\, {\rm Re}\,z_1>0\}.
$$
\end{thm}

In particular, Theorem \ref{thm2}  is applicable to pluriharmonic mappings.
In the case of planar harmonic mappings $f=h+\overline{g}$,
where $h$ and $g$ are analytic in the unit disk $\mathbb{D} := \{z\in \IC:\, |z| < 1 \}$, Theorem \ref{thm2} takes the following
form--another criterion for injectivity--harmonic analog of $\Phi$-like mappings, see \cite[Theorem 1]{GrafSamyVictor2015b}.

\begin{cor}\label{GPS2-th1}
Let $f=h+\overline{g}$ be harmonic on a convex domain $D\subset \IC$ and $\Omega=f(D)$. Then $f$ is univalent in $D$
if and only if there exists a complex-valued function $\phi=\phi(w,\overline{w})$ in $C^1(\Omega)$ and such that for every
$\epsilon$ with $|\epsilon|=1$, there exists a real number $\gamma=\gamma (\epsilon )$ satisfying
$${\rm Re}\,\big \{e^{i\gamma}\big(\partial\phi(f(z),\overline{f(z)})
+\epsilon \overline{\partial}\phi(f(z),\overline{f(z)})\big)\big \}>0 ~\mbox{ for all $z\in D$},
$$

where $\partial=\frac{\partial}{\partial z}$ and $\overline{\partial}=\frac{\partial}{\partial \overline{z}}$.
\end{cor}

Several consequences and examples of Corollary \ref{GPS2-th1} are discussed in \cite{GrafSamyVictor2015b} and they seem to be very useful.
Another univalence criterion for harmonic mappings of $\ID=\{z\in \IC:\, |z|<1\}$ was obtained in \cite{star2014}.
Moreover, using Corollary \ref{GPS2-th1}, it is easy to obtain the following sufficient condition for
the univalency of $C^1$ functions.

\bcor\label{GPS2-cor1}
{\rm (\cite{Mocanu80})}
Let $D$ be a convex domain in $\IC$ and $f$ be a complex-valued function of class $C^1(D)$. Then $f$ is univalent in $D$
if there exists a real number $\gamma$ such that
$$
{\rm Re}\,\big (e^{i\gamma} f_z(z) \big )> \big |f_{\overline{z}}(z)\big |  ~\mbox{ for all $z\in D$}.
$$
\ecor

For example, if $f=h+\overline{g}$ is a planar harmonic mapping in the unit disk $\ID$ and if there exists a real number $\gamma$ such that
\be\label{GPS2-eq2}
{\rm Re}\,\big \{e^{i\gamma} h'(z)\}>|g'(z)|~ \mbox{for $z\in \ID$,}
\ee
then $f$ is univalent in $\ID$. In \cite{Hiroshi-Samy-2010}, it was shown that harmonic functions $f=h+\overline{g}$ satisfying the
condition \eqref{GPS2-eq2} in $\ID$ are indeed univalent and \textit{close-to-convex} in $\ID$, i.e., the complement of the image-region $f(\ID)$
is the union of non-intersecting rays (except that the origin of one ray may lie on another one of the rays).

Moreover, using Theorem \ref{thm2}, one can also obtain a sufficient condition for $p$-valent mappings.

\begin{cor}\label{cor4}
Suppose that $D\subseteq\IR^n$ is a domain such that $D=\cup_{m=1}^pD_m,$ where $D_m$'s are convex for $m=1,\ldots,p$.
Furthermore, let $f\in C^1(D),$ $\Omega=f(D)\subset \IR^n$, $\phi\in C^1(\Omega)$ such that for every $X_m\in S_{n-1}$ there exists a
unitary matrix $U_m:=U_m(X_m)$ for each $m=1,\ldots,p,$ such that
$$U_m\cdot (D\phi)(f(X))\cdot (Df)(X)\cdot X_m\in\IR^n_{+}
$$
for every $X\in D_m,$ $m=1,\ldots,p.$  Then $f$ is no more than $p$-valent in $D$.
\end{cor}

\section{Proofs of Theorems \ref{cor4a}, \ref{thm5} and \ref{thm1a}} \label{sec3}

\begin{proof}[\bf Proof of Theorem \ref{cor4a}]
It is enough to prove the Theorem in the case  $A=B=I$.
For convenience, we let $z=x_1+\cdots +x_n$. Then
$$u_k(X)=x_k+\gamma_k\left[\alpha_2 z^2+\alpha_3 z^3+\cdots+\alpha_m z^m\right]
$$
for $k=1,\ldots,n$, $\alpha_j,\gamma_k\in\IR$ with $\sum_{k=1}^n\gamma_k=0$ and $X=(x_1,\ldots, x_n)$.

We first prove that $\det Df(X)\equiv 1.$ To do this, we begin to introduce
$$L(X)=2\alpha_2 z +3\alpha_3 z^2+\cdots+m\alpha_m z^{m-1}.
$$
Then a computation gives
$$\frac{\partial u_k}{\partial x_j}=\delta^j_k+\gamma_kL.
$$
and
$$\det Df=
\begin{vmatrix}1+\gamma_1L &\gamma_1L &\ldots &\gamma_1L\\
\gamma_2L & 1+\gamma_2L &\ldots &\gamma_2L\\
\vdots & \vdots &\vdots &\vdots\\
\gamma_nL&\gamma_nL&\ldots&1+\gamma_nL\end{vmatrix} =I_1+\gamma_1L I_2,
$$
where
$$I_1 = \begin{vmatrix}1&0&\ldots&0\\ \gamma_2L&1+\gamma_2L&\ldots&\gamma_2L\\
\vdots & \vdots &\vdots &\vdots\\
\gamma_nL&\gamma_nL&\ldots&1+\gamma_nL\end{vmatrix}
~\mbox{ and }~ I_2=\begin{vmatrix}1&1&\ldots&1\\
 \gamma_2L&1+\gamma_2L&\ldots&\gamma_2L\\
\vdots & \vdots &\vdots &\vdots\\
  \gamma_nL&\gamma_nL&\ldots&1+\gamma_nL\\\end{vmatrix}.
$$
We will now show by induction that
$$\det Df(X)=1+\left (\sum_{k=1}^{n}\gamma_k\right )L(X).
$$
Obviously, for $n=2$,  we have $\det Df(X)=1+(\gamma_1+\gamma_2)L$  for $\gamma_1,\gamma_2\in\IR$.

Next, we suppose that $\det Df(X)=1+(\gamma_1+\cdots+\gamma_{p})L$ holds for $p=n-1$, where
$\gamma_1,\ldots,\gamma_{p}\in\IR$. We need to show that it is true for $p=n$. Clearly,
$$I_1=\begin{vmatrix}1+\gamma _2 L&\ldots&\gamma_2 L\\
\vdots & \vdots &\vdots\\
\gamma_nL&\ldots&1+\gamma_nL \end{vmatrix}=1+(\gamma_2+\cdots+\gamma_n)L
$$
and
$$I_2=\begin{vmatrix}1&1&1&1&\ldots&1\\
0&1& 0&0&\ldots&0 \\
0& 0 &1 &0&\ldots &0\\
\vdots & \vdots &\vdots &\vdots &\vdots &\vdots\\
 0&0& 0& 0&\ldots &1\end{vmatrix}=1.
$$
Since $\det Df=I_1+\gamma_1L I_2$, using the above and the hypothesis $\sum_{k=1}^n\gamma_k=0$ that
$$\det Df(X)=1+(\gamma_1+\gamma_2+\cdots+\gamma_n)L(X)\equiv 1.
$$

Applying Theorem \ref{GPS3-lem1}, we will finally show that $f$ is indeed a univalent mapping.
Now, for convenience, we denote $L_*=\int_0^1L[\gamma(t)]\,dt$ and obtain that
\beqq
\det \left (\int_0^1\frac{\partial u_k}{\partial x_j}(\gamma (t) )\,dt\right )^n_{j,k=1} &= &
\begin{vmatrix}1+\gamma_1 L_*&\gamma_1 L_*&\ldots&\gamma_1 L_* \\
\gamma_2 L_*&1+\gamma_2 L_*&\ldots&\gamma_2 L_*\\
\vdots&\vdots&\vdots&\vdots\\
\gamma_n L_*&\gamma_n L_*&\ldots&1+\gamma_n L_*\\
\end{vmatrix}\\
&=& 1+(\gamma_1+\cdots+\gamma_n)L_*\\
&=&1
\eeqq
for all $\gamma(t)$ as in Theorem \ref{GPS3-lem1}. Thus, by Theorem \ref{GPS3-lem1}, $f$ is univalent in $\IR^n.$
\end{proof}

\begin{proof}[\bf Proof of Theorem \ref{thm5}]
Consider $f(X)=(u_1(X),u_2(X))$, where $X=(x,y)\in \IR^2$ and
$$u_1(x,y)=x+\left[\alpha_2(x+y)^2+\alpha_3(x+y)^3+\cdots+\alpha_m(x+y)^m\right]
$$
and
$$u_2(x,y)=y-\left[\alpha_2(x+y)^2+\alpha_3(x+y)^3+\cdots+\alpha_m(x+y)^m\right]
$$
for $\alpha_j\in\IR$, $j=2, \ldots n$. Let
$$\tilde{f}(X)=(u_1(X)+W(X),u_2(X)+w(X)),
$$
where $W$ and $w$ are as mentioned in the statement.

As in the proof of Theorem \ref{cor4a}, we see easily that
\beq \label{GPS2-eq3}
\nonumber \det D{\tilde{f}} & =&
\begin{vmatrix}1+L +W_x& L +W_y\\
-L +w_x& 1-L +w_y
\end{vmatrix}\\[2mm]
&= & 1+(1+ L)w_y+(1- L)W_x +LW_y - Lw_x+(W_xw_y-w_xW_y)
\eeq
which is identically $1$, by the hypothesis of Theorem \ref{thm5}.
Moreover, allowing $\|X\|\ra \infty$ in \eqref{GPS2-eq3}, it follows that $W_xw_y-w_xW_y=0$.
We now show that this gives the relation
$$w=\lambda W
$$
for some constant $\lambda$.
We observe that both $w$ and $W$ are not identically zero simultaneously. If $w\equiv 0$ and $W\not\equiv 0$, then we choose $\lambda =0$.
Because of the symmetry, equality holds in the last relation when $W\equiv 0$ and $w\not\equiv 0$. Therefore, it suffices to consider the case
$W\not\equiv 0 \not\equiv w$.
We denote $t=x/y$.  Using the definition of $w(x)$ and $W(X)$ in the statement, we may conveniently write
$$
w(X)=\sum_{k=0}^{m+1}\alpha_{k}x^ky^{m+1-k} = y^{m+1}\sum_{k=0}^{m+1}\alpha_{k}t^k =:  y^{m+1}p(t)
$$
and similarly,
\be\label{GPS2-eq5}
W(X)=\sum_{k=0}^{m+1}\beta_{k}x^ky^{m+1-k}=:  y^{m+1}q(t).
\ee
Using these, we find that
\beqq
w_x(X)&=&\sum_{k=1}^{m+1}k\alpha_{k}x^{k-1}y^{m+1-k}=y^{m}p'(t), ~\mbox{ and }\\
w_y(X)&=& \sum_{k=0}^{m}(m+1-k)\alpha_{k}x^ky^{m-k} 
=  y^{m}((m+1)p(t)-tp'(t)).
\eeqq
Similarly, we see that
$$W_x(X)= y^{m}q'(t)  ~\mbox{and }~W_y(X)= y^{m}((m+1)q(t)-tq'(t)).
$$
Then
\beqq
\frac{W_y(X)}{W_x(X)}= \frac{w_y(X)}{w_x(X)}
& \Longleftrightarrow &\frac{(m+1)q(t)-tq'(t)}{q'(t)} =\frac{(m+1)p(t)-tp'(t)}{p'(t)}\\
& \Longleftrightarrow & \frac{q'(t)}{q(t)} =\frac{p'(t)}{p(t)},
\eeqq
and the last relation, by integration, gives $q(t)=\lambda p(t)$ for some constant $\lambda$. Thus, we have
the desired claim $w=\lambda W$. Consequently, by \eqref{GPS2-eq3}, $\det D{\tilde{f}}(X)\equiv 1$ implies that
$$(1+ L)w_y+(1- L)W_x + LW_y - Lw_x\equiv 0
$$
which, by the relation $w=\lambda W$, becomes
\be \label{GPS2-eq4}
(\lambda + L\lambda + L)W_y+(1- L- L\lambda )W_x \equiv 0.
\ee
Allowing $\|X\|\ra 0$ in \eqref{GPS2-eq4}, we see that $\lambda W_y(X)+W_x(X)=0$ is equivalent to
$$\sum_{k=1}^{m+1}[\lambda (m+1-(k-1))\beta _{k-1}+k\beta _{k}]x^{k-1}y^{m+1-k} =0.
$$
This gives the condition $\lambda (m+2-k)\beta _{k-1}+k\beta _{k}=0$ for $k=1, \ldots, m+1$. From the last relation, it follows easily that
$$\beta _k=\frac{(-\lambda )^k (m+1)!}{k!(m+1-k)!}\beta _0 ~\mbox{ for $k=1, \ldots, m+1$}
$$
and thus, using this and  \eqref{GPS2-eq5}, we obtain that
$$W(X)=\beta_0(y-\lambda x)^{m+1}.
$$
Since
$$W_x(X)=-\lambda (m+1)\beta_0(y-\lambda x)^{m} ~\mbox{ and }~ W_y(X)=(m+1)\beta_0(y-\lambda x)^{m},
$$
allowing $\|X\|\ra \infty$ in \eqref{GPS2-eq4} (and making use of the expression $L$ in the proof of Theorem \ref{cor4a} for $n=2$), we have in case $L\not\equiv 0$ and
$\alpha _m \neq 0:$
$$\alpha_{m}(1+\lambda)^{2}=0,
$$
which gives the condition $\lambda =-1$.
Note that if $\alpha _m =0$, we choose the leading largest $j$ for which $\alpha _j\neq 0$. However, the last condition gives that
$\lambda =-1$. 
Consequently, we end up with the forms
$$W(X)=\beta_0(x+y)^{m+1}=\alpha_{m+1}(x+y)^{m+1}
$$
and
$$w(X)=-\beta_0(x+y)^{m+1}=-\alpha_{m+1}(x+y)^{m+1}.
$$

If $L\equiv 0$, then
$$\tilde{f}(X)=(x+\beta_0(y-\lambda x)^{m+1},y+\lambda \beta_0(y-\lambda x)^{m+1}).
$$
If at the same time $\lambda \neq 0$, then we denote
$$A= \left(\begin{array}{cccc}
\ds -\lambda  &  0 \\
\ds 0 & 1
\end{array}\right), ~ F(X)=A\tilde{f}(A^{-1}(X)) ~\mbox{ and }~ \alpha_{m+1}=-\lambda\beta_0.
$$
Thus, we have  $\tilde{f}=A^{-1} \circ F\circ A$ and
\be\label{PSJC-eq1}
F(X)=(x+\alpha_{m+1}(x+y)^{m+1},y -\alpha_{m+1}(x+y)^{m+1}).
\ee

If $L\equiv 0$ and $\lambda =0$, then $w \equiv 0$ and $W(X)=\beta_0y^{m+1}$ so that
$$\tilde{f}(X)=(x+\beta_0y^{m+1},y).
$$
Now,  we denote
$$A= \left(\begin{array}{cccc}
\ds 1  &  0 \\
\ds -1 & 1
\end{array}\right) ~ \mbox{ and }~ \alpha_{m+1}=\beta_0.
$$
This gives $\tilde{f}=A^{-1} \circ F\circ A$ and $F$ has the form \eqref{PSJC-eq1}.
The proof is complete.
\end{proof}

Proof of Theorem \ref{thm1a} requires some preparation.

\begin{lem}\label{lem2-new}
Suppose that  $G_j\in P_n(m)$ for  $j=1,\ldots,N$ and  $G_j(X)=X+u^{(j)}(X)\gamma^{(j)}$,  where  $u^{(j)}(X)=\sum_{l=2}^m a_l^{(j)}z^l$
with $z=x_1+\cdots +x_n$,  $\gamma^{(j)}=(\gamma^{(j)}_1,\ldots,\gamma^{(j)}_n)$,  $\sum_{k=1}^n\gamma^{(j)}_k=0$ for all $j$, and $a_l^{(j)}$ are
some constants. Then $G_1\circ \cdots \circ G_N=:F\in\mathcal{P}_n(m)$, and
$$F(X)=X+u^{(1)}(X)\gamma^{(1)}+\cdots +u^{(N)}(X)\gamma^{(N)}.
$$
\end{lem}
\begin{proof}

At first we would like to prove the lemma for $N=2$  and then extend it for the composition of $N$ mappings, $N\geq 2$.
Let $G_j\in P_n(m)$ for  $j=1,2$ and
$$G_j(X)=X+u^{(j)}(X)\gamma^{(j)}.
$$
Also, introduce
$$G_2(X)=(g^{(2)}_1,g^{(2)}_2, \ldots, g^{(2)}_n)
~\mbox{ and }~L_p(X)=\sum_{l=2}^m la_l^{(p)}z^{l-1}~\mbox{ for $p=1,2.$}
$$
We observe that
$$g^{(2)}_1(X)+  \cdots + g^{(2)}_n(X) =z + \sum_{l=2}^m \left [a_l^{(2)} \sum_{k=1}^n \gamma^{(2)}_k z^{l} \right ] =z
$$
and therefore,
$$L_1(G_2(X))= L_1(X) ~\mbox{ and }~D{G_1}(G_2(X))=D{G_1}(X).
$$
Finally, for the case of $N=2$, one has
$$DF(X)= D{G_1}(G_2(X))D{G_2}(X)= D{G_1}(X)D{G_2}(X)=A^{(1)}(L_1)A^{(2)}(L_2) =( r_{ij})_{n\times n},
$$
where $A^{(p)}(L_p)$ for $p=1,2$ are the two $n\times n$ matrices given by
$$A^{(p)}(L_p) = \left( {\begin{array}{cccc}
1+\gamma^{(p)}_1L_p & \gamma^{(p)}_1L_p  & \cdots & \gamma^{(p)}_1L_p\\
\gamma^{(p)}_2L_p & 1+\gamma^{(p)}_2L_p & \cdots & \gamma^{(p)}_2L_p\\
\vdots & \vdots & \vdots & \vdots \\
\gamma^{(p)}_nL_p & \gamma^{(p)}_nL_p & \cdots & 1+\gamma^{(p)}_nL_p\\
\end{array} } \right), \quad p=1,2.
$$


If $i\neq j$, then a computation gives
$$r_{ij}= \gamma_i^{(1)}L_1L_2\sum_{k=1}^n \gamma_k^{(2)}+\gamma_i^{(1)}L_1+\gamma_i^{(2)}L_2 = \gamma_i^{(1)}L_1+\gamma_i^{(2)}L_2 ;
$$
Similarly, for $i=j$, we have
$$r_{ii}= \gamma_i^{(1)}L_1L_2\sum_{k=1}^n \gamma_k^{(2)}+1+\gamma_i^{(1)}L_1+\gamma_i^{(2)}L_2=1+\gamma_i^{(1)}L_1+\gamma_i^{(2)}L_2.
$$
Consequently, we obtain that
$$DF(X)= \left( {\begin{array}{cccc}
   (1+\gamma^{(1)}_1L_1+ \gamma^{(2)}_1L_2) & \gamma^{(1)}_1L_1+ \gamma^{(2)}_1L_2  & \cdots & \gamma^{(1)}_1L_1+ \gamma^{(2)}_1L_2\\[2mm]
   \gamma^{(1)}_2L_1+\gamma^{(2)}_2L_2 & (1+\gamma^{(1)}_2L_1+\gamma^{(2)}_2L_2) & \cdots & \gamma^{(1)}_2L_1+\gamma^{(2)}_2L_2\\[2mm]
   \vdots & \vdots & \vdots & \vdots \\[2mm]
   \gamma^{(1)}_nL_1+\gamma^{(2)}_nL_2 & \gamma^{(1)}_nL_1+\gamma^{(2)}_nL_2 & \cdots & (1+\gamma^{(1)}_nL_1+\gamma^{(2)}_nL_2)\\
  \end{array} } \right).
$$
Moreover, it follows easily that
\begin{align*}
F(X)&=(G_1\circ G_2)(X)= \left( {\begin{array}{cc}
   x_1+\gamma^{(1)}_1u^{(1)}(X)+\gamma^{(2)}_1u^{(2)}(X) \\[2mm]
   x_2+\gamma^{(1)}_2u^{(1)}(X)+\gamma^{(2)}_2u^{(2)}(X) \\[2mm]
     \vdots  \\[2mm]
   x_n+\gamma^{(1)}_nu^{(1)}(X)+\gamma^{(2)}_nu^{(2)}(X)  \\[2mm]
  \end{array} } \right)  \\[2mm]
  & = X+u^{(1)}(X)\gamma^{(1)}+u^{(2)}(X)\gamma^{(2)}.
\end{align*}
Similarly in the case of the composition of three mappings $G_1$, $G_2$ and $G_3$ (i.e. for the case of $N=3$),
the Jacobian matrix of $F=G_1\circ G_2\circ G_3$ is given by
\beqq
DF(X)& =& D{G_1}((G_2\circ G_3)(X))D{G_2}(G_3(X))D{G_3}(X)\\
&=& D{G_1} (X ) D{G_2} (X) D{G_3} (X) \\
&=&( r_{ij})_{n\times n},
\eeqq
where,  for $i\neq j$,
$$r_{ij}=\gamma_i^{(1)}L_1+\gamma_i^{(2)}L_2+\gamma_i^{(3)}L_3, ~ L_3(X)=\sum_{l=2}^m la_l^{(3)}z^{l-1};
$$
and
$$r_{ii}=1+\gamma_i^{(1)}L_1+\gamma_i^{(2)}L_2+\gamma_i^{(3)}L_3.
$$
Moreover, we find that
$$F(X)=(G_1\circ G_2\circ G_3)(X)= X+u^{(1)}(X)\gamma^{(1)}+u^{(2)}(X)\gamma^{(2)}+u^{(3)}(X)\gamma^{(3)}.
$$
The above process may be continued to complete the proof.
\end{proof}

\br
We remark that the lemma is of interest only when the dimension of the space is greater than or equal to $3$.
For $n=2$, it does not give anything new in comparison with Theorem \ref{cor4a}.
However, for $n>2$, we obtain from Lemma \ref{lem2-new} new mappings from $\mathcal{P}_n(m)$, not belonging
to $P_n (m)$.
\er

\begin{example}
 Let $m=3=n$, $N=2$, $a_l^{(1)}=l-1$, $a_l^{(2)}=2+l$, $\gamma^{(1)}=(1,2,-3)$,
$\gamma^{(2)}=(-3,1,2)$ and apply Lemma \ref{lem2-new} for the mappings $G_1,G_2\in P_3(3)$.

Define  $F(X)=(G_1\circ  G_2)(X)=(u_1(X),u_2(X),u_3(X))$, where $X=(x_1,x_2,x_3)$, $z=x_1+x_2+x_3$. Since
\beqq
F(X)& =& X+u^{(1)}(X)\gamma^{(1)}+u^{(2)}(X)\gamma^{(2)}\\
& =& X+ (z^2+2z^3)(1,2,-3)+(4z^2+5z^3)(-3,1,2),
\eeqq
a comparison gives
$$u_1(X)=x_1-11z^2-13z^3,~ u_2(X)=x_2+6z^2+9z^3 ~\mbox{ and }~u_3(X)=x_3+5z^2+4z^3.
$$
According to Lemma \ref{lem2-new}, we obtain that  $F\in \mathcal{P}_3(3)$. We now show that  $F\notin P_3(3)$. For its proof,
it is enough to show that there are no such vectors $\Gamma=(\Gamma_1,\Gamma_2)$
and $A=(A_1,A_2)$ from $\IR^2$ that
$$\Gamma_1 A=(-11,-13),~ \Gamma_2 A=(6,9)~\mbox{ and }~ (-\Gamma_1-\Gamma_2)A=(5,4).
$$
It is easy to see that the above system of equations has no solution which means that  $F\notin P_3(3)$. \hfill{$\Box$}
\end{example}

Now, we prove our final result which offers functions from $\mathcal{P}_n(m)$ and generalizes Theorem \ref{cor4a} significantly
in a natural way.

\begin{proof}[\bf Proof of Theorem \ref{thm1a}]
The idea of the proof is in presenting $f$ as a composition of mappings $G_j\in P_n(m)$ for $j=1,\ldots,N$.
For $G_j$, we may write
$G_j(X)=(u_1^{(j)},\ldots,u_n^{(j)}),$
where
$$u_k^{(j)}(X)=x_k+\gamma_k^{(j)}\sum_{l=2}^{m}A_l^{(j)}z^l,\quad  \sum_{k=1}^n\gamma_k^{(j)}=0  ~\mbox{ for all $j=1, \ldots , N$,}
$$
and  $A_l^{(j)}$ are some constants.

From Lemma \ref{lem2-new}, the possibility of choosing $f=G_1 \circ \cdots \circ G_N$ means the following:
there exist two sets of numbers
$$\left\{\gamma_k^{(j)}\right\}^{n,\,N}_{k=1,\,j=1} \mbox{ with } \sum_{k=1}^n\gamma_k^{(j)}=0~\mbox{ for all $j=1, \ldots , N$,}
$$
and
$$\left \{A_l^{(j)}\right \}^{m,\,N}_{l=2,\,j=1}
$$
such that  for all $k=1, \ldots , n$,
\be\label{eq1-ex1}
\gamma_k^{(1)}A_l^{(1)}+\gamma_k^{(2)}A_l^{(2)}+\cdots+\gamma_k^{(N)}A_l^{(N)}=p_k^{(l)}
\ee
holds for each $l=2, \ldots , m$.
For the solution of this task, for every $j=1,\ldots,N$, we fix a set of nonzero numbers $\left\{\gamma_k^{(j)}\right \}^{n}_{k=1}$ such that
$\sum_{k=1}^n\gamma_k^{(j)}=0$ and the rank of the matrix $\left \{\gamma_k^{(j)}\right \}^{n,\,N}_{k=1,\,j=1}$ is equal to $(n-1)$~(we can consider $N>n$).
Then in each of the linear system of equations \eqref{eq1-ex1} for $l=2,\ldots,m$, the corresponding matrix
$\left (\gamma_k^{(j)}\right )^{n,\,N}_{k=1,\,j=1}$ of the system is one and the same and has also the rank $(n-1)$, and $A_l^{(1)},\dots,A_l^{(N)}$ play a
role of variables in the system of equations \eqref{eq1-ex1}. At the same time
in each system \eqref{eq1-ex1}, the rank of an expanded matrix $\left (\gamma_k^{(j)}\cup p_k^{(l)}\right )^{n,\,N}_{k=1,\,j=1}$
as well as the rank of $\left (\gamma_k^{(j)}\right )^{n,\,N}_{k=1,\,j=1}$ will be equal to $(n-1)$ in each system, since
$$\sum_{k=1}^np_k^{(l)}=0=\sum_{k=1}^n\gamma_k^{(j)} \quad \mbox{ for each $l$ and $j$}.
$$
Therefore, according to the theorem of Kronecker, each of the system of equations \eqref{eq1-ex1}  $(l=2,\ldots,m)$ will have the solution
$(A_l^{(1)},\ldots,A_l^{(N)})$. It finishes the proof of Theorem \ref{thm1a}.
\end{proof}

\br
For $n\ge 3$, Theorem~\ref{thm1a} significantly expands the set $P_n(m)$ of mappings for which \textbf{JC} is fair. Really,
without parameters of matrices $A$ and $B$ from Theorem \ref{cor4a}, the  set $P_n(m)$ has $[(m-1)+(n-1)]$ free parameters,
and the set of polynomial mappings from Theorem~\ref{thm1a} has  $(m-1)(n-1)$ such parameters.
\er

\subsection*{Acknowledgements}
The work of V.V. Starkov is supported Russian Science Foundation under grant 17-11-01229 and performed in Petrozavodsk State University.
The first author  is currently on leave and is working at ISI Chennai Centre, Chennai, India.



\end{document}